\documentclass{amsart}

\usepackage{latexsym}
\usepackage{amssymb, latexsym}
\usepackage{amsmath}
\usepackage{mathrsfs}
\usepackage{amsxtra}
\usepackage{diagrams}
\usepackage{amsthm}
\usepackage{verbatim}
\newtheorem{theorem}{Theorem}[section]

\newtheorem{lemma}[theorem]{Lemma}
\newtheorem{proposition}[theorem]{Proposition}

\newtheorem*{main}{Main~Theorem}
\theoremstyle{definition}
\newtheorem{definition}[theorem]{Definition}
\newtheorem{question}[theorem]{Question}

\newcommand{\Erdos}{Erd\H os}
\newcommand{\Konig}{K\"{o}nig}

\newcommand{\image}{\mathbin{\hbox{\tt\char'42}}}

\newcommand{\Union}{\bigcup}
\newcommand{\union}{\cup}
\newcommand{\of}{\subseteq}
\newcommand{\lt}[1]{{\smalllt}#1}
\newcommand{\lesseq}[1]{{\smallleq}#1}
\newcommand{\smallleq}{\mathrel{\mathchoice{\raise2pt\hbox{$\scriptstyle\leq$}}{\raise1pt\hbox{$\scriptstyle\leq$}}{\raise1pt\hbox{$\scriptscriptstyle\leq$}}{\scriptscriptstyle\leq}}}
\newcommand{\smalllt}{\mathrel{\mathchoice{\raise2pt\hbox{$\scriptstyle<$}}{\raise1pt\hbox{$\scriptstyle<$}}{\raise0pt\hbox{$\scriptscriptstyle<$}}{\scriptscriptstyle<}}}
\newcommand{\Add}{\mathop{\rm Add}}
\newcommand{\ltkappa}{{{\smalllt}\kappa}}

\newcommand{\GCH}{{\rm GCH}}
\newcommand{\ORD}{\mathop{{\rm ORD}}}
\newcommand{\ZFC}{{\rm ZFC}}

\newcommand{\one}{\mathop{1\hskip-2.5pt {\rm l}}}

\newcommand{\p}{\mathbb{P}}
\newcommand{\q}{\mathbb{Q}}

\newcommand{\la}{\langle}
\newcommand{\ra}{\rangle}

\newcommand{\tail}{\text{tail}}
\newcommand{\forces}{\Vdash}
\newcommand{\restrict}{\upharpoonright}
\newcommand{\supp}{\text{supp}}
\newcommand{\Coll}{\mathop{\rm Coll}}
\newcommand{\from}{\mathbin{\vbox{\baselineskip=2pt\lineskiplimit=0pt
                          \hbox{.}\hbox{.}\hbox{.}}}}

\begin{document}
\title{Indestructibility properties of remarkable cardinals}
\author[Yong Cheng]{Yong Cheng}
\address[Yong Cheng]{Institute of Mathematical Logic and Fundamental Re-
search, Einsteinstr 62, 48149 Muenster, Germany}
\email[Y. ~Cheng]{world-cyr@hotmail.com}

\author[Victoria Gitman]{Victoria Gitman}
\address[Victoria Gitman]{
CUNY Graduate Center, 365 Fifth Avenue, New York, NY, 10016, USA}
\email[V. ~Gitman]{vgitman@nylogic.org}
\urladdr{http://boolesrings.org/victoriagitman/}

\maketitle
\begin{abstract}
Remarkable cardinals were introduced by Schindler, who showed that the existence of a remarkable cardinal is equiconsistent with the assertion that the theory of $L(\mathbb R)$ is absolute for proper forcing~\cite{schindler:remarkable1}. Here, we study the indestructibility properties of remarkable cardinals. We show that if $\kappa$ is remarkable, then there is a forcing extension in which the remarkability of $\kappa$ becomes indestructible by all $\ltkappa$-closed $\lesseq\kappa$-distributive forcing and all two-step iterations of the form \hbox{$\Add(\kappa,\theta)*\dot{\mathbb R}$}, where $\dot{\mathbb R}$ is forced to be $\ltkappa$-closed and $\lesseq\kappa$-distributive. In the process, we introduce the notion of a \emph{remarkable Laver function} and show that every remarkable cardinal carries such a function. We also show that remarkability is preserved by the canonical forcing of the $\GCH$.
\end{abstract}
\section{Introduction}
Since the seminal results of Levy and Solovay~\cite{levysolovay:ch} on the indestructibility of large cardinals by small forcing and Laver on making a supercompact cardinal $\kappa$ indestructible by all $\lt\kappa$-directed closed forcing~\cite{laver:supercompact}, indestructibility properties of various large cardinal notions have been intensively studied. A decade after Laver's result, Gitik and Shelah showed that a strong cardinal $\kappa$ can be made indestructible by all weakly $\lesseq\kappa$-closed forcing with the Prikry property~\cite{GitikShelah:IndestructibleStrongCardinals}, a class that includes all $\lesseq\kappa$-closed forcing, and Woodin showed, using his technique of surgery, that it can be made indestructible by forcing of the form $\Add(\kappa,\theta)$.\footnote{The poset $\Add(\kappa,\theta)$, where $\kappa$ is an infinite cardinal and $\theta$ is any cardinal, adds $\theta$-many Cohen subsets to $\kappa$, using conditions of size less than $\kappa$.} More recently, Hamkins and Johnstone showed that a strongly unfoldable cardinal $\kappa$ can be made indestructible by all $\lt\kappa$-closed $\kappa^+$-preserving forcing~\cite{hamkinsjohnstone:unfoldable}. It turns out that not all large cardinals possess robust indestructibility properties. Very recently, Bagaria et al. showed that a number of large cardinal notions including superstrong, huge, and rank-into-rank cardinals are superdestructible: such $\kappa$ cannot even be indestructible by $\Add(\kappa,1)$~\cite{BagariaHamkinsTsaprounisUsuba:SuperstrongAndOtherLargeCardinalsAreNeverLaverIndestructible}. In this article, we show that remarkable cardinals have indestructibility properties resembling those of strong cardinals. A remarkable $\kappa$ be made indestructible by all $\ltkappa$-closed $\lesseq\kappa$-distributive forcing and by all two-step iterations of the form $\Add(\kappa,\theta)*\dot{\mathbb R}$, where $\dot{\mathbb R}$ is forced to be $\ltkappa$-closed and $\lesseq\kappa$-distributive.
\begin{main}
If $\kappa$ is remarkable, then there is a forcing extension in which the remarkability of $\kappa$ becomes indestructible by all $\ltkappa$-closed $\lesseq\kappa$-distributive forcing and by all two-step iterations $\Add(\kappa,\theta)*\dot{\mathbb R}$, where $\dot{\mathbb R}$ is forced to be $\ltkappa$-closed and $\lesseq\kappa$-distributive.
\end{main}
\noindent In particular, a remarkable $\kappa$ can be made indestructible by all $\lesseq\kappa$-closed forcing, and since $\dot {\mathbb R}$ can be trivial, by all forcing of the form $\Add(\kappa,\theta)$. One application of the main theorem is that any $\GCH$ pattern can be forced above a remarkable cardinal. Another application uses a recent forcing construction of~\cite{ChengFriedmanHamkins:LargeCardinalsNeedNotBeLargeInHOD}, to produce a remarkable cardinal that is not weakly compact in ${\rm HOD}$. Using techniques from the proof of the main theorem, we also show that remarkability is preserved by the canonical forcing of the $\GCH$.

For the indestructibility arguments, we define the notion of a \emph{remarkable Laver function} and show that every remarkable cardinal carries such a function. Although Laver-like functions can be forced to exist for many large cardinal notions~\cite{Hamkins:LaverDiamond}, remarkable cardinals along with supercompact and strong cardinals are some of the few that outright possess such fully set-anticipating functions. For instance, not every strongly unfoldable cardinal has a Laver-like function because it is consistent that $\diamondsuit_\kappa({\rm REG})$ fails at a strongly unfoldable cardinal~\cite{DzamonjaHamkins2006:DiamondCanFail}.

Schindler introduced remarkable cardinals when he isolated them as the large cardinal notion whose existence is equiconsistent with the assertion that the theory of $L(\mathbb R)$ cannot be altered by proper forcing. The assertion that the theory of $L(\mathbb R)$ is absolute for all set forcing is intimately connected with ${\rm AD}^{L(\mathbb R)}$ and its consistency strength lies in the neighborhood of infinitely many Woodin cardinals~\cite{schindler:remarkable1}. In contrast, remarkable cardinals are much weaker than measurable cardinals, and indeed they can exist in $L$. Consistency-wise, they fit tightly into the $\alpha$-iterable hierarchy of large cardinal notions (below a Ramsey cardinal) introduced by Gitman and Welch~\cite{gitman:welch}, where they lie above 1-iterable, but below 2-iterable cardinals, placing them above hierarchies of ineffability, but much below an $\omega$-\Erdos\ cardinal. Remarkable cardinals are also totally indescribable and $\Sigma_2$-reflecting. Strong cardinals are remarkable, but the least measurable cardinal cannot be remarkable  by $\Sigma_2$-reflection.
\begin{definition}
A cardinal $\kappa$ is \emph{remarkable} if in the $\Coll(\omega,\ltkappa)$ forcing extension $V[G]$, for every regular cardinal $\lambda>\kappa$, there is a $V$-regular cardinal $\overline\lambda<\kappa$ and $j:H_{\overline\lambda}^V\to H_\lambda^V$ with critical point $\gamma$ such that $j(\gamma)=\kappa$.
\end{definition}
\noindent Schindler originally used a different primary characterization of remarkable cardinals (see Theorem~\ref{th:equivalentRemarkable} in the next section), but he has recently switched to using the above characterization, which gives remarkable cardinals a character of generic supercompactness~\cite{schindler:remarkable}. Magidor showed that $\kappa$ is supercompact if and only if for every regular cardinal $\lambda>\kappa$, there is a regular cardinal $\overline \lambda<\kappa$ and  $j:H_{\overline\lambda}\to H_\lambda$ with critical point $\gamma$ such that  $j(\gamma)=\kappa$~\cite{magidor:supercompact}.

More background material on remarkable cardinals is presented in Section~\ref{sec:preliminaries}. Remarkable Laver functions are introduced in Section~\ref{sec:LaverFunction}. All the indestructibility results are proved in Section~\ref{sec:mainTheorem}, and some applications of indestructibility are given in Section~\ref{sec:applications}. The final Section~\ref{sec:questions} lists some remaining open questions concerning indestructibility properties of remarkable cardinals.
\section{Preliminaries}\label{sec:preliminaries}
Most indestructibility arguments for large cardinals rely on their characterizations in terms of the existence of some kind of elementary embeddings $j:M\to N$ between transitive models of (fragments of) set theory. The large cardinal property is verified in a forcing extension $V[G]$ by \emph{lifting}, meaning extending, the embedding $j$ to \hbox{$j:M[G]\to N[H]$}. The success of this strategy is based on the \emph{Lifting Criterion} theorem (Proposition 9.1 in \cite{cummings:handbook}) that provides a sufficient condition for the existence of an $N$-generic filter $H$ necessary to carry out the lift.
\begin{theorem}[Lifting Criterion]
Suppose that $M$ is a model of $\ZFC^-$ and \hbox{$j:M\to N$} is an elementary embedding.\footnote{The theory $\ZFC^-$ consists of the axioms of $\ZFC$ excluding the powerset axiom and including collection instead of replacement. Canonical models of $\ZFC^-$ are $H_\theta$, where $\theta$ is a regular cardinal.} If $M[G]$ and $N[H]$ are generic extensions by forcing notions $\p$ and $j(\p)$ respectively, then the embedding $j$ lifts to an elementary embedding $j:M[G]\to N[H]$ with $j(G)=H$ if and only if $j\image G\of H$.
\end{theorem}
\noindent The following simple proposition is used often in lifting arguments to obtain a generic filter satisfying the lifting criterion.
\begin{proposition}\label{prop:cpSizeImage}
Suppose that $M$ and $N$ are transitive models of $\ZFC^-$ and \hbox{$j:M\to N$} is an elementary embedding with critical point $\gamma$. If $A$ has size $\gamma$ in $M$, then $j\image A\in N$.
\end{proposition}
\begin{proof}
Fix a bijection $f:\gamma\xrightarrow[\text{onto}]{1-1} A$ in $M$. It is easy to see that $j\image A=j(f)\image\gamma$, and both $j(f)$ and $\gamma$ are elements of $N$.
\end{proof}

Getting back to remarkable cardinals, let's recall that a cardinal $\kappa$ is remarkable if in the $\Coll(\omega,\ltkappa)$ forcing extension $V[G]$, for every regular cardinal $\lambda>\kappa$, there is a $V$-regular cardinal $\overline\lambda<\kappa$ and $j:H_{\overline\lambda}^V\to H_\lambda^V$ with critical point $\gamma$ such that $j(\gamma)=\kappa$.\footnote{In the remainder of the article, $H_\theta$ and $V_\theta$ will always refer to ground model objects, a convention which allows us to drop the superscript $V$. The $H_\theta$ and $V_\theta$ of a forcing extension $V[G]$ will be referred to by $H_\theta^{V[G]}$ and $V_\theta^{V[G]}$ respectively.} This characterization of remarkable cardinals admits several useful generalizations. For the remainder of this discussion, we suppose that $G\subseteq \Coll(\omega,\ltkappa)$ is $V$-generic.

Fix some regular $\lambda>\kappa$. In $V[G]$, let $j:H_{\overline\lambda}\to H_\lambda$, where $\overline\lambda<\kappa$ is $V$-regular, $\text{cp}(j)=\gamma$, and $j(\gamma)=\kappa$. Let $G_\gamma$ be the restriction of $G$ to the sub-product $\prod_{\xi<\gamma}\Coll(\omega,\xi)$ of $\Coll(\omega,\ltkappa)$. Now observe that, by the lifting criterion, we can lift $j$ to the elementary embedding
\begin{displaymath}
j:\la H_{\overline\lambda}[G_\gamma],H_{\overline\lambda}\ra\to\la H_\lambda[G],H_\lambda\ra\text{ with }j(G_\gamma)=G.
\end{displaymath}
Let's make the following general definition to eliminate a cumbersome repetition of hypothesis from future arguments.
\begin{definition}
In a $\Coll(\omega,\ltkappa)$-forcing extension $V[G]$, let us say that an elementary embedding
\begin{displaymath}
j:H_{\overline\lambda}\to H_\lambda
\end{displaymath}
is $(\mu,\overline\lambda,\xi,\lambda)$-\emph{remarkable}
if $\lambda>\xi$ and $\overline\lambda<\xi$ are $V$-regular, $\text{cp}(j)=\mu$, and $j(\mu)=\xi$. Let us also say that
\begin{displaymath}
j:\la H_{\overline\lambda}[G_\mu],H_{\overline\lambda}\ra\to\la H_\lambda[G_\xi],H_\lambda\ra
\end{displaymath}
is $(\mu,\overline\lambda,\xi,\lambda)$-\emph{very remarkable} if $j\restrict H_{\overline\lambda}:H_{\overline\lambda}\to H_\lambda$ is remarkable and $j(G_\mu)=G_\xi$.
\end{definition}
\noindent By definition, every $(\mu,\overline\lambda,\xi,\lambda)$-very remarkable embedding restricts to a $(\mu,\overline\lambda,\xi,\lambda)$-remarkable embedding, and, by the lifting criterion, every $(\mu,\overline\lambda,\xi,\lambda)$-remarkable embedding lifts to a $(\mu,\overline\lambda,\xi,\lambda)$-very remarkable embedding.

Fixing a regular $\lambda>\kappa$, we can ask whether there are $(\gamma,\overline\lambda,\kappa,\lambda)$-remarkable embeddings in $V[G]$ with $\gamma$ arbitrarily high in $\kappa$. More generally, we can ask whether any $a\in H_\lambda$ must appear in the image of some $(\gamma,\overline\lambda,\kappa,\lambda)$-remarkable embedding, because if $\xi<\kappa$ appears in the range of a $(\gamma,\overline\lambda,\kappa,\lambda)$-remarkable $j$, then $\gamma>\xi$. A positive answer follows from the following stronger result. Let us say that $X\subseteq H_\lambda[G]$ is $(\kappa,\lambda)$-\emph{remarkable} if it is the range of some $(\gamma,\overline\lambda,\kappa,\lambda)$-very remarkable embedding. Note that $(\kappa,\lambda)$-remarkable $X$ are countable in $V[G]$.
\begin{proposition}[\cite{schindler:remarkable}]\label{prop:StatManyEmbeddings}
If $\kappa$ is remarkable, then in a $\Coll(\omega,\ltkappa)$ forcing extension $V[G]$, for every regular cardinal $\lambda>\kappa$, the collection of all $(\kappa,\lambda)$-remarkable $X$ is stationary in $[H_\lambda[G]]^\omega$.
\end{proposition}
\begin{proof}
Suppose to the contrary that there is a regular cardinal $\lambda>\kappa$ such that the set of all $(\kappa,\lambda)$-remarkable $X$ is not stationary in $[H_\lambda[G]]^\omega$ and assume that $\lambda$ is the least such. A large enough $H_\delta[G]$ will see that this is the case and $\lambda$ will be definable there. So we fix a $(\gamma,\overline\delta,\kappa,\delta)$-very remarkable embedding
\begin{displaymath}
j:\la H_{\overline\delta}[G_\gamma], H_{\overline\delta}\ra\to \la H_\delta[G],H_\delta\ra
\end{displaymath}
with $j(\overline\lambda)=\lambda$. By elementarity, $H_{\overline\delta}[G_\gamma]$ thinks that there is some club $\overline C$ of $[H_{\overline\lambda}[G_\gamma]]^\omega$ without any $(\gamma,\overline\lambda)$-remarkable $X$. Let
\begin{displaymath}
F_{\overline C}:[H_{\overline\lambda}[G_\gamma]]^{\lt\omega}\to H_{\overline\lambda}[G_\gamma]
\end{displaymath}
in $H_{\overline\delta}[G_\gamma]$ be such that all $X$ closed under $F_{\overline C}$ are in $\overline C$ (by Theorem 8.28 in \cite{jech:settheory}). Let $j(F_{\overline C})=F$. By elementarity, $H_\delta[G]$ satisfies that no $(\kappa,\lambda)$-remarkable $X$ is closed under $F$. Now consider the obvious elementary embedding (see Proposition~\ref{prop:remarkableEmbeddingRestriction} below)
\begin{displaymath}
j:\la H_{\overline\lambda}[G_\gamma],G_\gamma,F_{\overline C}\ra\to \la H_\lambda[G], G,F\ra,
\end{displaymath}
and let $X=j\image H_{\overline\lambda}[G_\gamma]$. Clearly $X$ is $(\kappa,\lambda)$-remarkable and $X$ is closed under $F$. Also clearly $X\in H_\delta[G]$. So we have reached a contradiction showing that the collection of $(\kappa,\lambda)$-remarkable $X$ must be stationary in $[H_\lambda[G]]^{\omega}$.
\end{proof}
It is a simple but very handy observation that if $j$ is a $(\gamma,\overline\delta,\kappa,\delta)$-remarkable or very remarkable embedding and $\lambda>\kappa$ is a regular cardinal in the range of $j$ with $j(\overline\lambda)=\lambda$, then $j$ restricts to a $(\gamma,\overline\lambda,\kappa,\lambda)$-remarkable or very remarkable embedding respectively.
\begin{proposition}\label{prop:remarkableEmbeddingRestriction}
In a $\Coll(\omega,\ltkappa)$-forcing extension $V[G]$, if $j$ is $(\gamma,\overline\delta,\kappa,\delta)$-very remarkable and $\lambda>\kappa$ is regular with $j(\overline\lambda)=\lambda$, then $j$ restricts to a $(\gamma,\overline\lambda,\kappa,\lambda)$-very remarkable embedding
\begin{displaymath}
j:\la H_{\overline\lambda}[G_\gamma],H_{\overline\lambda}\ra\to \la H_\lambda[G],H_\lambda\ra.
\end{displaymath}
The same statement holds for remarkable embeddings.
\end{proposition}
\begin{proof}
First, observe that $\overline\lambda$ is $V$-regular. Since $H_\delta$ satisfies that $\lambda$ is regular, $H_{\overline\delta}$ satisfies that $\overline\lambda$ is regular, but it must be correct about this because it has all sets of hereditary size $\lt\overline\delta$.
Next, observe that $H_{\overline\lambda}$ is definable in $\la H_{\overline\delta}[G_\gamma],H_{\overline\delta}\ra$ as the collection of all sets of hereditary size $\lt\overline\lambda$ in $H_{\overline\delta}$, and the same formula with $\lambda$ in place of $\overline\lambda$ defines $H_\lambda$ in $\la H_{\delta}[G],H_{\delta}\ra$. It follows that if $\la H_{\overline\lambda}[G_\gamma],H_{\overline\lambda}\ra\models\varphi(a)$, then this is first-order expressible in the structure $\la H_{\overline\delta}[G_\gamma], H_{\overline\delta}\ra$, and therefore by elementarity, $\la H_{\delta}[G],H_\delta\ra$ satisfies that $\la H_\lambda[G], H_\lambda\ra\models\varphi(j(a))$. This argument also proves the statement for remarkable embeddings.
\end{proof}

Every $(\gamma,\overline\lambda,\kappa,\lambda)$-remarkable embedding has an extender-like factor embedding that will be used in the indestructibility arguments.
\begin{proposition}\label{prop:remarkableExtender}
Suppose that $\kappa$ is remarkable and $\lambda>\kappa$ is regular. In a $\Coll(\omega,\ltkappa)$-forcing extension $V[G]$, if $j:H_{\overline\lambda}\to H_\lambda$ is a $(\gamma,\overline\lambda,\kappa,\lambda)$-remarkable embedding, then there is an embedding $j^*:H_{\overline\lambda}\to N$ and an embedding $h:N\to H_\lambda$ such that
\begin{enumerate}
\item $N=\{j^*(f)(a)\mid a\in S^{\lt\omega},\,f\in H_{\overline\lambda}\}$ is transitive, where $S=V_\kappa\union\{\kappa\}$,
\item $\text{cp}(h)>\kappa$,
\item the following diagram commutes:
\begin{diagram}
H_{\overline\lambda} &\rTo^j&H_\lambda\\
 \dTo^{j^*}&\ruTo^h &\\
N & &
\end{diagram}
\end{enumerate}
\end{proposition}
\begin{proof}
Let $X=\{j(f)(a)\mid a\in S^{\lt\omega},f\in H_{\overline\lambda}\}$, where $S=V_\kappa\union \{\kappa\}$, and let $\pi:
X\to N$ be the Mostowski collapse of $X$. It is easy to see that $X\prec H_\lambda$ and  $j\image H_{\overline\lambda}\subseteq X$. Thus, we define $j^*(x)=\pi\circ j(x)$ and $h=\pi^{-1}$. Since $S\subseteq X$, it follows that the critical point of $h$ must be above $\kappa$.
\end{proof}
\noindent If $j$ and $j^*$ are as in Proposition~\ref{prop:remarkableExtender}, we shall say that $j^*$ is the \emph{remarkable extender} embedding obtained from $j$.

Analysis of properties of remarkable cardinals often relies on the following folklore result about the absoluteness of existence of elementary embeddings of countable models.
\begin{lemma}[Absoluteness Lemma for countable embeddings]\label{le:absolutenessLemma}
Suppose that $M\models\ZFC^-$ is countable and $j:M\to N$ is an elementary embedding. If $W\subseteq V$ is a transitive (set or class) model of $\ZFC^-$ such that $M$ is countable in $W$ and $N\in W$, then $W$ has some elementary embedding $j^*:M\to N$. Moreover, if both $M$ and $N$ are transitive, $\text{cp}(j)=\gamma$ and $j(\gamma)=\delta$, we can additionally assume that $\text{cp}(j^*)=\gamma$ and $j^*(\gamma)=\delta$. Also, we can assume that $j$ and $j^*$ agree on some fixed finite number of values.
\end{lemma}
\begin{proof}
Fix some enumeration $M=\la a_i\mid i\in\omega\ra$ that exists in $W$. Say that a sequence $\langle b_0,\ldots,b_{n-1}\rangle$ of elements of $N$ is a \emph{finite partial embedding} if for every formula $\varphi(x_0,\ldots,x_{n-1})$, we have that
\begin{displaymath}
M\models\varphi(a_0,\ldots,a_{n-1})\Leftrightarrow N\models\varphi(b_0,\ldots,b_{n-1}).
\end{displaymath}
Let $T$ consist of all finite partial embeddings. Observe that $T$ is clearly a tree under the natural ordering by extension and it has height at most $\omega$. Also, $T$ is an element of $W$. Clearly any infinite branch through $T$ gives an embedding $h:M\to N$. Let $c_i=j(a_i)$. Every sequence $\langle c_0,\ldots,c_{n-1}\rangle$ is a partial finite embedding in $W$ and the collection $\langle c_i\mid i\in\omega\rangle$ is an infinite branch through $T$ in $V$. Thus, the tree $T$ is ill-founded in $V$. But then $T$ must be ill-founded in $W$ by the absoluteness of well-foundedness. Thus, $W$ has an infinite branch through $T$ and this branch gives some embedding $j^*:M\to N$. To achieve that $j$ and $j^*$ agree on the critical point, we limit the tree $T$ to finite partial embeddings fixing all $\alpha$ below the critical point of $j$ and having $b_n=\delta$, where $a_n=\gamma$. An agreement on some finitely many values is achieved similarly.
\end{proof}
We will often use the fact that $\Coll(\omega,\ltkappa)$ is a weakly homogeneous poset\footnote{A poset $\p$ is said to be \emph{weakly homogeneous} if for any two conditions $p,q\in\p$, there is an automorphism $\pi$ of $\p$ such that $p$ and $\pi(q)$ are compatible.}. So, in particular, if $p\in \Coll(\omega,\ltkappa)$ forces some statement $\varphi$ that involves only check names, then $\one_{\Coll(\omega,\ltkappa)}\forces \varphi$.

We end the preliminaries with an alternative characterization of remarkable cardinals that was originally Schindler's primary definition.
\begin{theorem}[Schindler \cite{schindler:remarkable2}]\label{th:equivalentRemarkable}
A cardinal $\kappa$ is remarkable if and only if for every regular cardinal $\lambda>\kappa$, there are countable transitive models $M$ and $N$ with embeddings
\begin{enumerate}
\item $\pi:M\to H_\lambda$ with $\pi(\overline\kappa)=\kappa$,
\item $\sigma:M\to N$ such that
\begin{enumerate}
\item  $\text{cp}(\sigma)=\overline\kappa$,
\item $\ORD^M$ is a regular cardinal in $N$ and $M=H_{\ORD^M}^N$,
\item $\sigma(\overline\kappa)>\ORD^M$.
\end{enumerate}
\end{enumerate}
\end{theorem}
\begin{proof}$\,$\\
($\Rightarrow$): Suppose that $\kappa$ is remarkable. Fix a regular cardinal $\lambda>\kappa$ and let $\delta=(2^\lambda)^+$. Let $X$ be any countable elementary substructure of $H_\delta$ with $\kappa,\lambda\in X$ and let $\overline M$ be the Mostowski collapse of $X$.  The inverse of the collapse is an elementary embedding
\begin{displaymath}
\overline\pi:\overline M\to H_\delta
\end{displaymath}
with $\overline\pi(\kappa')=\kappa$ and $\overline\pi(\lambda')=\lambda$. Since $\kappa$ is remarkable, $H_\delta$ satisfies that every $\Coll(\omega,\ltkappa)$ forcing extension has some embedding $j:H_\beta\to H_\lambda$, where $\beta$ is $V$-regular and the critical point is sent to $\kappa$. Thus, $\overline M$ satisfies the same statement for $\kappa'$ and $\lambda'$ by elementarity. Since $\overline M$ is countable, we can choose an $\overline M$-generic $g$ for $\Coll(\omega,\lt\kappa')^{\overline M}$ and $\overline M[g]$ must have an $\overline M$-regular $\beta<\kappa'$ and
\begin{displaymath}
\sigma: H_{\beta}^{\overline M}\to H_{\lambda'}^{\overline M}
\end{displaymath}
with $\text{cp}(\sigma)=\overline\kappa$ and $\sigma(\overline\kappa)=\kappa'$. Let $M=H_{\beta}^{\overline M}$ and let $\pi=\overline\pi\circ\sigma$. Now it is easy to see that $\pi$ and $\sigma$ are the required embeddings.\\
\noindent ($\Leftarrow$): Fix a regular cardinal $\lambda>\kappa$ and let $\delta>2^\lambda$ be regular such that $\lambda$ is definable in $H_\delta$ (for instance $\delta=\gamma^+$ where $\gamma>2^\lambda$ has cofinality $\lambda$). Let
\begin{displaymath}
\pi:M\to H_\delta\text{ and }\sigma:M\to N
\end{displaymath}
be as in the statement of the theorem with $\pi(\overline\lambda)=\lambda$. Since $M$ is countable, we can choose an $M$-generic filter $g$ for $\Coll(\omega,\lt\overline\kappa)^M$. Because $M=H_{\ORD^M}^N$ and $N$ is countable, we can choose some $g'$ extending $g$ that is $N$-generic for $\Coll(\omega,\lt\sigma(\overline\kappa))^N$. So, by the lifting criterion, we can lift $\sigma$ to
\begin{displaymath}
\sigma:M[g]\to N[g'].
\end{displaymath}
Consider now the restriction
\begin{displaymath}
\sigma:H^M_{\overline\lambda}\to H^N_{\sigma(\overline\lambda)}.
\end{displaymath}
Notice that $H_{\overline\lambda}^M$ is countable in $N[g']$. By the absoluteness lemma~(\ref{le:absolutenessLemma}), $N[g']$ must have some embedding
\begin{displaymath}
\tau:H^M_{\overline\lambda}\to H_{\sigma(\overline\lambda)}^N\text{ with }\text{cp}(\tau)=\overline\kappa\text{ and } \tau(\overline\kappa)=\sigma(\overline\kappa).
\end{displaymath}
 Thus, $N[g']$ satisfies that there exists an $N$-regular $\beta<\sigma(\overline\kappa)$, $\gamma<\beta$ and
\begin{displaymath}
\tau:H_{\beta}\to H_{\sigma(\overline\lambda)}\text{ with }\text{cp}(\tau)=\gamma\text{ and }\tau(\gamma)=\sigma(\overline\kappa).
\end{displaymath}
So by elementarity, $M[g]$ satisfies that there exists an $M$-regular $\beta<\overline\kappa$, $\gamma<\beta$ and
\begin{displaymath}
\tau:H_{\beta}\to H_{\overline\lambda}\text{ with }\text{cp}(\tau)=\gamma\text{ and }\tau(\gamma)=\overline\kappa.
\end{displaymath}
Since $\Coll(\omega,\ltkappa)$ is weakly homogeneous, $M$ satisfies that this statement is forced by $\Coll(\omega,\lt\overline\kappa)$. But then by elementarity, $H_{\delta}$ satisfies that it is forced by $\Coll(\omega,\ltkappa)$ that there exists a regular $\beta<\kappa$, $\gamma<\beta$ and
\begin{displaymath}
\tau:H_{\beta}\to H_{\lambda}\text{ with }\text{cp}(\tau)=\gamma\text{ and  }\tau(\gamma)=\kappa.
\end{displaymath}
\end{proof}
\noindent In the future, we will call a pair of embeddings $\pi$ and $\sigma$ satisfying the conditions of Theorem~\ref{th:equivalentRemarkable} a $(\kappa,\lambda)$-\emph{remarkable pair}.
\section{A remarkable Laver function}\label{sec:LaverFunction}
Laver defined and used a Laver function on a supercompact cardinal $\kappa$ to show that it can be made indestructible by all $\lt\kappa$-directed closed forcing. Very generally, given a large cardinal $\kappa$ characterized by the existence of some kind of elementary embeddings, a Laver-like function $\ell$ on $\kappa$ has the property that for any $a$ in the universe, $\ell$ \emph{anticipates} $a$ in the sense that there is an embedding $j$, of the type characterizing the large cardinal, such that $j(\ell)(\kappa)=a$. Although the existence of Laver-like functions can be forced for many large cardinals, only a few large cardinals such as supercompact, strong, and extendible cardinals have them outright \cite{GitikShelah:IndestructibleStrongCardinals, corazza:laver}. We shall define a Laver-like function for a remarkable cardinal and prove that every remarkable cardinal carries such a function.
\begin{definition}
Suppose that $\kappa$ is remarkable, $\ell\from\xi\to V_\kappa$ for some $\xi\leq\kappa$, and $G\subseteq\Coll(\omega,\ltkappa)$ is $V$-generic.\footnote{The symbol $\from$ is used to indicate a possibly partial function.} We shall say that $x$ is $\lambda$-\emph{anticipated} by $\ell$ (for $V$-regular $\lambda$) in $V[G]$, if $x\in H_\lambda$ and $V[G]$ has a $(\mu,\overline\lambda,\xi,\lambda)$-remarkable embedding $h:H_{\overline\lambda}\to H_{\lambda}$ such that
\begin{enumerate}
\item $\ell\restrict\mu+1\in H_{\overline\lambda}$,
\item $\mu\in\text{dom}(\ell)$,
\item $h(\ell\restrict\mu+1)(\xi)=x$.
\end{enumerate}
\end{definition}
\begin{definition}
Suppose that $\kappa$ is remarkable. We define that a function $\ell\from\kappa\to V_\kappa$ has the \emph{remarkable Laver} property if whenever $\lambda>\kappa$ is regular and $G\subseteq\Coll(\omega,\ltkappa)$ is $V$-generic, every $x\in H_\lambda$ is $\lambda$-anticipated by $\ell$ in $V[G]$.
We define that $l$ is a \emph{remarkable Laver} function if it has the remarkable Laver property and for every $\xi\in\text{dom}(\ell)$, we have that $\xi$ is inaccessible and $\ell\image\xi\subseteq V_\xi$.
\end{definition}
\noindent We construct a remarkable Laver function by adapting Laver's construction to the context of remarkable cardinals. Let's start with some preliminaries.

\begin{lemma}\label{le:UnanticipatedSetInVkappa}
Suppose that $\kappa$ is remarkable and $\ell\from\xi\to V_\kappa$ for some $\xi<\kappa$. In a $\Coll(\omega,\ltkappa)$-forcing extension $V[G]$, if there is a regular $\lambda$ for which some set is not $\lambda$-anticipated by $\ell$, then the least such $\lambda$ is below $\kappa$.
\end{lemma}
\begin{proof}
Fix some $x$ that is not $\lambda$-anticipated by $\ell$. Since $\kappa$ is inaccessible, $\ell$ is an element of some $V_\beta$ for $\beta<\kappa$. In $V[G]$, choose $H_\tau[G]$ large enough so that it can see that $x$ is not $\lambda$-anticipated by $\ell$, and let
\begin{displaymath}
j:\la H_{\overline\tau}[G_\gamma], H_{\overline\tau}\ra\to \la H_\tau[G],H_\tau\ra
\end{displaymath}
be a $(\gamma,\overline\tau,\kappa,\tau)$-very remarkable embedding such that $j(\ell)=\ell$ and $x,\lambda\in\text{ran}(j)$ (Proposition~\ref{prop:StatManyEmbeddings}). Let $j(\overline x)=x$ and $j(\alpha)=\lambda$. By elementarity, $\la H_{\overline\tau}[G_\gamma], H_{\overline\tau}\ra$ satisfies that $\overline x$ is not $\alpha$-anticipated by $\ell$. We claim that the structure $\la H_{\overline\tau}[G_\gamma], H_{\overline\tau}\ra$ must be correct about this. Suppose towards a contradiction that $\overline x$ is $\alpha$-anticipated by $\ell$. Then there is some $(\mu,\overline\alpha,\xi,\alpha)$-remarkable embedding $h:H_{\overline\alpha}\to H_\alpha$ such that
\begin{enumerate}
\item $\ell\restrict\mu+1\in H_{\overline\alpha}$,
\item $\mu\in\text{dom}(\ell)$,
\item $h(\ell\restrict\mu+1)(\xi)=\overline x$.
\end{enumerate}
Since $\overline\alpha<\gamma$ and $\gamma$ is inaccessible by elementarity, we have that $H_{\overline\alpha}$ is already countable in $V[G_\gamma]$. By using the absoluteness lemma between $V[G_\gamma]$ and $V[G]$, we have that there is a $(\mu,\overline\alpha,\xi,\alpha)$-remarkable embedding $h':H_{\overline\alpha}\to H_\alpha$ in $V[G_\gamma]$ with $h'(\ell\restrict\mu+1)(\xi)=\overline x$. Since $h'\in V[G_\gamma]$, it must already be an element of $H_{\overline\tau}[G_\gamma]$, which contradicts our assumption that there is no such embedding there.

Thus, we found a set, namely $\overline x$, that is not $\alpha$-anticipated by $\ell$ with $\alpha<\kappa$.
\end{proof}
\begin{lemma}\label{le:AnticipationIsAbsolute}
Suppose that $\kappa$ is remarkable and $\ell\from\xi\to V_\kappa$ for some $\xi<\kappa$. If $G$ and $H$ are two $V$-generic filters for $\Coll(\omega,\ltkappa)$, then $V[G]$ and $V[H]$ must agree on the least $\lambda$ for which some $x$ is not $\lambda$-anticipated by $\ell$, and they must agree on the collection of such $x$.
\end{lemma}
\begin{proof}
Suppose that in $V[G]$, $a$ is not $\lambda$-anticipated by $\ell$, where $\lambda$ is least for which some set is not $\lambda$-anticipated. Let $\varphi$ be a sentence in the forcing language expressing this situation and note that the only $\Coll(\omega,\ltkappa)$-names $\varphi$ contains are check names. Since $\Coll(\omega,\ltkappa)$ is weakly homogeneous, $\one_{\Coll(\omega,\ltkappa)}\forces\varphi$ and so $V[H]$ must agree that $\varphi$ holds.
\end{proof}
\begin{definition}\label{def:laver}
Suppose that $\kappa$ is remarkable and $W$ is some well-ordering of $V_\kappa$ of order-type $\kappa$. We define a partial function $\ell_W\from\kappa\to V_\kappa$ inductively as follows. Suppose that $\ell_W\restrict\xi$ has been defined. If there is $\lambda$ such that
\begin{center}
$\one_{\Coll(\omega,\ltkappa)}\forces$ there is a set that is not $\lambda$-anticipated by $\ell_W\restrict\xi$,
\end{center}
 then $\ell_W(\xi)$ is the $W$-least $a$ such that
\begin{center}
$\one_{\Coll(\omega,\ltkappa)}\forces \lambda$ is least for which a set is not $\lambda$-anticipated by $\ell_W\restrict\xi$ and $a$ is not $\lambda$-anticipated by $\ell_W\restrict\xi$.
\end{center}
Otherwise, $\ell_W(\xi)$ is undefined.
\end{definition}
\noindent Note that, by Lemma~\ref{le:UnanticipatedSetInVkappa}, if there is a $\lambda$ for which some set is not $\lambda$-anticipated by $\ell$, then the least such $\lambda$ is below $\kappa$ and therefore there will always be a witnessing set in the range of $W$, namely $V_\kappa$. Note also that, using Lemma~\ref{le:AnticipationIsAbsolute}, we can define $\ell_W$ directly in any $\Coll(\omega,\ltkappa)$-extension $V[G]$. Suppose that $\ell_W\restrict \xi$ has been defined. If there is $\lambda$ for which some set is not $\lambda$-anticipated, then we let $\ell_W(\xi)$ be the $W$-least $a$ such that $\lambda$ is the least with that property and $a$ is not $\lambda$-anticipated. Otherwise, $\ell_W(\xi)$ is undefined. In the remainder of the article, whenever we mention $\ell_W$, we will always tacitly assume that $W$ is a well-ordering of $V_\kappa$ of order-type $\kappa$.
\begin{proposition}\label{prop:preimageOfLaverFunction}
Suppose that $\kappa$ is remarkable. In a $\Coll(\omega,\ltkappa)$-forcing extension $V[G]$, if $j:H_{\overline\lambda}\to H_\lambda$ is a $(\gamma,\overline\lambda,\kappa,\lambda)$-remarkable embedding with $W\in\text{\emph{ran}}(j)$, then $j(\ell_W\restrict\gamma)=\ell_W$.
\end{proposition}
\begin{proof}
By replacing $j$ with its lift, we can assume that  $j$ is $(\gamma, \overline{\lambda}, \kappa, \lambda)$-very remarkable. Since $\kappa,W\in\text{ran}(j)$ and $\ell_W$ is definable from these over $H_\lambda[G]$, it follows that $\ell_W\in\text{ran}(j)$. So let $j(\overline W)=W$ and $j(\overline \ell)=\ell_W$. By elementarity, $\overline \ell\from \gamma\to V_\gamma$ is definable over $H_{\overline\lambda}[G_\gamma]$ precisely as $\ell_W$ with respect to the well-ordering $\overline W$, which is itself an initial segment of $W$ that well-orders $V_\gamma$ in order-type $\gamma$. The conclusion now follows because, by the absoluteness lemma, $H_{\overline\lambda}[G_\gamma]$ must agree with $V[G]$ about which sets are not anticipated by the initial segments of $\ell_W$.
\end{proof}
We now come to the main theorem of this section.
\begin{theorem}
If $\kappa$ is remarkable, then $\ell_W$ has the remarkable Laver property, and by restricting the domain of $\ell_W$ we obtain a remarkable Laver function.
\end{theorem}
\begin{proof}
Suppose that $G\subseteq\Coll(\omega,\ltkappa)$ is $V$-generic. We work in $V[G]$ and suppose towards a contradiction that $\ell_W$ does not have the remarkable Laver property. Let $\lambda>\kappa$ be the least regular cardinal such that there is $x\in H_\lambda$ that is not $\lambda$-anticipated by $\ell_W$. Choose $H_\tau[G]$ large enough so that it can see this. Let
\begin{displaymath}
j:\la H_{\overline\tau}[G_\gamma],H_{\overline\tau}\ra\to \la H_\tau[G],H_\tau\ra
\end{displaymath}
be a $(\gamma,\overline\tau,\kappa,\tau)$-very remarkable embedding with $W\in\text{ran}(j)$ (Proposition~\ref{prop:StatManyEmbeddings}). By Proposition~\ref{prop:preimageOfLaverFunction},  $j(\ell_W\restrict\gamma)=\ell_W$. Observe that $\lambda\in\text{ran}(j)$ because it is definable as the least for which some set is not $\lambda$-anticipated by $\ell_W$. So let $j(\overline\lambda)=\lambda$. Also, since $\text{ran}(j)$ is elementary in $H_\tau[G]$, there must be some $x=j(\overline x)$ that is not $\lambda$-anticipated by $\ell_W$.

First, we argue that $\gamma\in\text{dom}(\ell_W)$ and $\ell_W\restrict \gamma+1$ is an element of $H_{\overline\tau}[G_\gamma]$. By elementarity, $H_{\overline\tau}[G_\gamma]$ satisfies that $\overline x$ is not $\overline\lambda$-anticipated by $\ell_W\restrict\gamma$, and it must be correct about this by the absoluteness lemma. Thus, $\ell_W$ is defined at $\gamma$, and moreover $\ell_W(\gamma)\in H_{\overline\lambda}$. Now we can consider $j(\ell_W\restrict\gamma+1)(\kappa)$. Because $H_\tau[G]$ knows that $\lambda$ is the least for which there is a set that is not $\lambda$-anticipated by $\ell_W$, by elementarity upwards, $j(\ell_W\restrict\gamma+1)(\kappa)=y$, where $y$ is some set that $H_{\tau}[G]$ thinks is not $\lambda$-anticipated by $\ell_W$. Now, using Proposition~\ref{prop:remarkableEmbeddingRestriction}, we restrict $j$ to a $(\gamma,\overline\lambda,\kappa,\lambda)$-very remarkable embedding. Call this restriction $j'$. Since $\ell_W\restrict\gamma\from \gamma\to V_\gamma$ and $\ell_W(\gamma)\in H_{\overline\lambda}$, it follows that $\ell_W\restrict\gamma+1\in H_{\overline\lambda}$. So $j'(\ell_W\restrict\gamma+1)(\kappa)=y$ and, by size considerations, $j'\in H_\tau[G]$. But now we have reached a contradiction because this means that $y$ was indeed $\lambda$-anticipated by $\ell_W$ and $H_\tau[G]$ can see this.

Finally, it remains to observe that by restricting the domain of $\ell_W$, we can assume without loss of generality, that it is defined only at inaccessible cardinals $\xi$ such that $\ell_W\restrict \xi\subseteq V_\xi$.
\end{proof}
A remarkable Laver function $\ell$ is needed in indestructibility arguments because whenever an iterated forcing $\p_\kappa$ of length $\kappa$ is defined to have nontrivial stages only for values in the domain of $\ell$, in the $\Coll(\omega,\ltkappa)$-forcing extension $V[G]$, we can find for any regular $\lambda$, a $(\gamma,\overline\lambda,\kappa,\lambda)$-remarkable embedding $j$ such that there is no forcing in $\p_\kappa$ in the interval $(\gamma,\overline\lambda]$.
\begin{lemma}\label{le:LaverFunctionGap}
Suppose that $\kappa$ is remarkable and $\ell_W$ is a remarkable Laver function. In a $\Coll(\omega,\ltkappa)$-forcing extension $V[G]$, for every regular cardinal $\lambda>\kappa$, there is a $(\gamma,\overline\lambda,\kappa,\lambda)$-remarkable embedding $j$ such that
\begin{enumerate}
\item $(\gamma,\overline\lambda]\cap \text{\emph{dom}}(\ell_W)=\emptyset$,
\item $\ell_W(\gamma)$ is defined,
\item $j(\ell_W\restrict\gamma)=\ell_W$.
\end{enumerate}
Given any $a,b\in H_\lambda$, we can additionally assume that
\begin{itemize}
\item[(4)] $a,b\in \text{ran}(j)$,
\item[(5)] $\ell_W(\gamma)=\la \overline a,x\ra$ for some set $x$, where $j(\overline a)=a$.
\end{itemize}
\end{lemma}
\begin{proof}
Since $\ell_W$ is a remarkable Laver function, we can fix a $(\gamma,\overline\delta,\kappa,\delta)$-remarkable embedding $j:H_{\overline\delta}\to H_\delta$ such that
\begin{enumerate}
\item $\ell_W\restrict\gamma+1\in H_{\overline\delta}$,
\item $\gamma\in\text{dom}(\ell_W)$,
\item $j(\ell_W\restrict\gamma+1)(\kappa)=\la \lambda+1,W\ra$.
\end{enumerate}
Our choice of $j(\ell_W\restrict\gamma+1)(\kappa)$ places both $\lambda$ and $W$ in the range of $j$. Thus, $j(\ell_W\restrict\gamma)=\ell_W$ by Proposition~\ref{prop:preimageOfLaverFunction}. Let $j(\overline\lambda)=\lambda$. Since $j(\ell_W\restrict\gamma+1)(\kappa)\notin V_\lambda$, by elementarity, it follows that $\ell_W(\gamma)\notin V_{\overline\lambda}$. But $\ell_W$ is a remarkable Laver function and so $\ell_W\restrict\xi\subseteq V_\xi$ for all $\xi$ in domain of $\ell_W$. Thus, $\ell_W$ cannot have anything in its domain between $\gamma$ and  $\overline\lambda$.  It follows that the restriction of $j$ to a $(\gamma,\overline\lambda,\kappa,\lambda)$-remarkable embedding (using Proposition~\ref{prop:remarkableEmbeddingRestriction}) has all the desired properties.

For the additional conclusions, we just modify
\begin{displaymath}
j(\ell_W\restrict\gamma+1)(\kappa)=\la a,\la\lambda+1,W,b\ra\ra.
\end{displaymath}
\end{proof}

\section{Indestructible remarkable cardinals}\label{sec:mainTheorem}
The indestructibility properties of remarkable cardinals closely resemble those of strong cardinals, of which they are generally viewed as a miniature version. At the conclusion of this section, we will show that if $\kappa$ is remarkable, then there is a forcing extension in which its remarkability becomes indestructible by all $\ltkappa$-closed $\leq\kappa$-distributive forcing and by all two-step iterations of the form $\Add(\kappa,\theta)*\dot{\mathbb R}$, where $\dot{\mathbb R}$ is forced to be $\ltkappa$-closed and $\lesseq\kappa$-distributive. We will also show that remarkability is preserved by the canonical forcing of the $\GCH$.
\subsection{Small forcing}
It is straightforward to see that remarkable cardinals are indestructible by small forcing.
\begin{proposition}
Remarkable cardinals are indestructible by small forcing.
\end{proposition}
\begin{proof}
Suppose that $\kappa$ is remarkable and fix a poset $\p$ such that $|\p|<\kappa$. By considering an isomorphic copy of $\p$, if necessary, we can assume that $\p\in V_\kappa$. Let $g\subseteq\p$ be $V$-generic and let $H\subseteq\Coll(\omega,\ltkappa)$ be $V[g]$-generic. We need to show that $V[g][H]$ has a $(\gamma,\overline\lambda,\kappa,\lambda)$-remarkable embedding for every regular $\lambda>\kappa$.

Fix a regular $\lambda>\kappa$. Since the definition of $\Coll(\omega,\ltkappa)$ is absolute, $V[H]$ is a $\Coll(\omega,\ltkappa)$-forcing extension. So in $V[H]$, we can fix a $(\gamma,\overline\lambda,\kappa,\lambda)$-remarkable embedding $j:H_{\overline\lambda}\to H_\lambda$ with $\gamma$ above the rank of $\p$ (by Proposition~\ref{prop:StatManyEmbeddings}). Clearly $j\in V[g][H]$ as well. By the lifting criterion, $j$ lifts to
\begin{displaymath}
j: H_{\overline\lambda}[g]\to H_\lambda[g]
\end{displaymath}
in $V[g][H]$. But clearly, because $\p$ is small relative to $\overline\lambda$, we have that $H_{\overline\lambda}[g]=H_{\overline\lambda}^{V[g]}$, $H_\lambda[g]=H_\lambda^{V[g]}$, and $\overline\lambda$ remains regular in $V[g]$.
\end{proof}
\subsection{Indestructibility by $\Add(\kappa,1)$}
As a warm-up theorem to the more general results, let's show that a remarkable cardinal $\kappa$ can be made indestructible by $\Add(\kappa,1)$.
\begin{theorem}
If $\kappa$ is remarkable, then there is a forcing extension in which the remarkability of $\kappa$ becomes indestructible by $\Add(\kappa,1)$.
\end{theorem}
\begin{proof}
Let's fix a remarkable Laver function $\ell_W$. Let $\p_\kappa$ be the $\kappa$-length Easton support iteration that forces with $\Add(\xi,1)^{V^{\p_\xi}}$ at stages $\xi$ in the domain of $\ell_W$, whenever $\xi$ remains a cardinal in $V^{\p_\xi}$. We will argue that the remarkability of $\kappa$ is indestructible by $\Add(\kappa,1)$ in any forcing extension by $\p_\kappa*\Add(\kappa,1)$. To show this, it suffices to argue that $\kappa$ remains remarkable after forcing with $\p_\kappa*\Add(\kappa,1)$ because $\Add(\kappa,1)\times\Add(\kappa,1)\cong\Add(\kappa,1)$.

So suppose that $G*g\subseteq\p_\kappa*\Add(\kappa,1)$ is $V$-generic and $H\subseteq\Coll(\omega,\ltkappa)$ is $V[G][g]$-generic. We need to show that $V[G][g][H]$ has a $(\gamma,\overline\lambda,\kappa,\lambda)$-remarkable embedding for every regular $\lambda>\kappa$.

Fix a regular $\lambda>\kappa$. In $V[H]$, we fix, using Lemma~\ref{le:LaverFunctionGap}, a $(\gamma,\overline\lambda,\kappa,\lambda)$-remarkable embedding $j:H_{\overline\lambda}\to H_\lambda$ such that
\begin{enumerate}
\item $(\gamma,\overline\lambda]\cap \text{dom}(\ell_W)=\emptyset$,
\item $\ell_W(\gamma)$ is defined,
\item $j(\ell_W\restrict\gamma)=\ell_W$.
\end{enumerate}
It is easy to see that $\p_\kappa$ preserves all inaccessible cardinals. By elementarity, $\gamma$ is inaccessible, and so, in particular, it remains a cardinal after forcing with $\p_\gamma$. Thus, there is forcing at stage $\gamma$ in $\p_\kappa$.

Let $G_\gamma*g_\gamma$ be the restriction of $G$ to $\p_\gamma*\Add(\gamma,1)$. Since $j(\ell_W\restrict\gamma)=\ell_W$, we have $j(\p_\gamma)=\p_\kappa$. Thus, by the lifting criterion, $j$ lifts to
\begin{displaymath}
j:H_{\overline\lambda}[G_\gamma]\to H_\lambda[G]
\end{displaymath}
in $V[G][g][H]$. Next, we lift $j$ to $H_{\overline\lambda}[G_\gamma][g_\gamma]$. Observe that $j\image g_\gamma=g_\gamma$ and $p=\bigcup g_\gamma$ is a condition in $\Add(\kappa,1)^{V[G]}$. The lifting criterion is not satisfied outright because there is no reason to suppose that $p\in g$, but this is easily fixed. Let $\pi$ be an automorphism of $\Add(\kappa,1)^{V[G]}$ in $V[G]$, which switches $g\restrict\gamma$ with $p$. Recall that $\pi\image g$ is $V[G]$-generic for $\Add(\kappa,1)^{V[G]}$ and $V[G][g]=V[G][\pi\image g]$. Thus, by replacing $g$ with $\pi\image g$ if necessary, we can assume that $p\in g$. Thus, we can lift $j$ to
\begin{displaymath}
j:H_{\overline\lambda}[G_\gamma][g_\gamma]\to H_\lambda[G][g]
\end{displaymath}
in $V[G][g][H]$. Since there is no forcing in $\p_\kappa$ in the interval $(\gamma,\overline\lambda]$, it follows that $H_{\overline\lambda}[G_\gamma][g_\gamma]=H_{\overline\lambda}^{V[G][g]}$. Also, clearly $H_\lambda[G][g]=H_\lambda^{V[G][g]}$. Finally, $\overline\lambda$ remains regular in $V[G][g]$ because $\p_\gamma*\Add(\gamma,1)$ has size $\gamma$, and therefore cannot affect the regularity of $\overline\lambda$, and the next forcing in the iteration is above $\overline\lambda$.
\end{proof}

\subsection{Indestructibility by $\Add(\kappa,\theta)$}

We can greatly generalize the result of the previous section, by employing more sophisticated techniques, to show that a remarkable $\kappa$ can be made simultaneously indestructible by all posets $\Add(\kappa,\theta)$.
\begin{theorem}\label{th:addCohenSubsetsIndestructible}
If $\kappa$ is remarkable, then there is a forcing extension in which the remarkability of $\kappa$ becomes indestructible by all forcing of the form $\Add(\kappa,\theta)$ for a cardinal $\theta$.
\end{theorem}
\begin{proof}
Let's fix a remarkable Laver function $\ell_W$. Let $\p_\kappa$ be the $\kappa$-length Easton support iteration that forces with $\Add(\xi,\mu)^{V^{\p_\xi}}$ at stages $\xi$ such that $l_W(\xi)=\la \mu,x\ra$ for some set $x$, whenever $\xi$ and $\mu$ are cardinals in $V^{\p_\xi}$. We will argue that $\kappa$ has the desired indestructibility in any forcing extension by $\p_\kappa$.

First, we argue that $\kappa$ remains remarkable in any forcing extension by $\p_\kappa$. Suppose that $G\subseteq\p_\kappa$ is $V$-generic and $H\subseteq \Coll(\omega,\ltkappa)$ is $V[G]$-generic. We need to show that $V[G][H]$ has a $(\gamma,\overline\lambda,\kappa,\lambda)$-remarkable embedding for every regular $\lambda>\kappa$.

Fix a regular $\lambda>\kappa$. In $V[H]$, we fix, using Lemma~\ref{le:LaverFunctionGap}, a $(\gamma,\overline\lambda,\kappa,\lambda)$-remarkable embedding
\begin{displaymath}
j:H_{\overline\lambda}\to H_\lambda
\end{displaymath}
such that
\begin{enumerate}
\item $(\gamma,\overline\lambda]\cap \text{dom}(\ell_W)=\emptyset$,
\item $j(\ell_W\restrict\gamma)=\ell_W$,
\item $\ell_W(\gamma)=\la \overline a,x\ra$, where $\overline a$ is not an ordinal.
\end{enumerate}
Observe that there is no forcing in $\p_\kappa$ at stage $\gamma$ because $\ell_W(\gamma)$ does not have the required form.  Thus, there is no forcing in $\p_\kappa$ at stages in $[\gamma,\overline\lambda]$. Since $j(\ell_W\restrict\gamma)=\ell_W$, we have $j(\p_\gamma)=\p_\kappa$. Thus, by the lifting criterion, $j$ lifts to
\begin{displaymath}
j:H_{\overline\lambda}[G_\gamma]\to H_\lambda[G]
\end{displaymath}
in $V[G][H]$. Since there is no forcing in $\p_\kappa$ on the interval $[\gamma,\overline\lambda]$, it follows that $H_{\overline\lambda}[G_\gamma]=H_{\overline\lambda}^{V[G]}$. Also, clearly $H_{\lambda}[G]=H_{\lambda}^{V[G]}$ and $\overline\lambda$ remains regular in $V[G]$. This completes the argument that $\kappa$ is remarkable in $V[G]$.

Fix a cardinal $\theta>\kappa$. Next, we argue that $\kappa$ is remarkable in any forcing extension by $\p_\kappa*\Add(\kappa,\theta)$. We will use the characterization of remarkable cardinals given in Theorem~\ref{th:equivalentRemarkable} and show how to lift a $(\kappa,\lambda)$-remarkable pair of embeddings by combining results about remarkable extender embeddings from  Proposition~\ref{prop:remarkableExtender} with Woodin's technique of surgery.

So let's suppose towards a contradiction that $\kappa$ is not remarkable in a $\p_\kappa*\Add(\kappa,\theta)$ forcing extension. Then there is a regular $\lambda>\kappa$ and a condition $q\in \p_\kappa*\Add(\kappa,\theta)$ forcing that there is no $(\kappa,\lambda)$-remarkable pair of embeddings $\pi$ and $\sigma$. We can assume that $\lambda\!>\!>\!\theta$ because if there are embeddings for arbitrarily large $\lambda'$, then we can always find some such embedding with $\lambda$ in the range and restrict.

Let $\delta=(2^\lambda)^+$. Let $Y$ be a countable elementary substructure of $H_\delta$ containing $q$, $\kappa$, $\theta$, $\lambda$, and the well-order $W$, and let $M'$ be the Mostowski collapse of $Y$. The inverse of the collapsing map is an embedding
\begin{displaymath}
\rho:M'\to H_\delta
\end{displaymath}
with $\rho(q')=q$, $\rho(\kappa')=\kappa$, $\rho(\theta')=\theta$, $\rho(\lambda')=\lambda$, and $\rho(\ell'_{W'})=\ell_W$. In $M'$, we can define the Easton support $\kappa'$-length iteration  $\overline \p_{\kappa'}$ that forces with $\Add(\xi,\mu)^{(M')^{\overline \p_\xi}}$ at stages $\xi$ such that $\ell'_{W'}(\xi)=\la \mu,x\ra$ for some set $x$, whenever $\xi$ and $\mu$  are cardinals in $(M')^{\overline\p_\xi}$. Clearly $\rho(\overline\p_{\kappa'})=\p_\kappa$. Since $M'$ is countable, we can choose some $M'$-generic $h$ for $\Coll(\omega,\lt\kappa')^{M'}$. By elementarity (and Lemma~\ref{le:LaverFunctionGap}), $M'[h]$ has a $(\overline\kappa,\alpha,\kappa',\lambda')$-remarkable embedding
\begin{displaymath}
\overline\sigma:H_\alpha^{M'}\to H_{\lambda'}^{M'}
\end{displaymath}
with $\overline\sigma(\overline q)=q'$, $\overline\sigma(\overline \ell)=\ell'_{W'}$, and $\overline\sigma(\overline\theta)=\theta'$ such that
\begin{enumerate}
\item $\text{dom}(\ell'_{W'})\cap (\overline\kappa,\alpha]=\emptyset$,
\item $\ell'_{W'}(\overline\kappa)=\la \overline\theta,x\ra$ for some set $x$,
\item $\overline \ell=\ell'_{W'}\restrict\overline\kappa$.
\end{enumerate}
In particular, we have $\overline\sigma(\overline\p_{\overline\kappa})=\overline\p_{\kappa'}$ and $\overline\p_{\kappa'}$ forces with $\Add(\overline\kappa,\overline\theta)^{(M')^{\overline\p_{\overline\kappa}}}$ at stage $\overline\kappa$. Let
\begin{displaymath}
\sigma:H^{M'}_\alpha\to N
\end{displaymath}
be the remarkable extender embedding obtained from $\overline\sigma$. Note that $H^{M'}_\alpha\subseteq N$ and $\overline\sigma$ and $\sigma$ agree on $V^{M'}_{\overline\kappa+1}$. Thus, for instance, it continues to be the case that $\sigma(\overline \ell)=\ell'_{W'}$ and $\sigma(\overline\p_{\overline\kappa})=\overline\p_{\kappa'}$. Let $M=H_\alpha^{M'}$ and
\begin{displaymath}
\pi:M\to H_\lambda,
\end{displaymath}
where $\pi=\rho\circ\overline\sigma$. Note that $\pi(\overline\kappa)=\kappa$ and $\pi(\overline q)=q$. Clearly $\pi$ and $\sigma$ constitute a $(\kappa,\lambda)$-remarkable pair of embeddings.

Since $\p_\kappa*\Add(\kappa,\theta)$ is in particular countably closed, we can find a $V$-generic filter $G*g\subseteq \p_\kappa*\Add(\kappa,\theta)$ that is $X$-generic for $X=\pi\image M$ with $q\in G*g$. Let $\overline G*\overline g$ be the pre-image of $X\cap (G*g)$ under $\pi$. Since $G*g$ is $X$-generic, it follows that $\overline G*\overline g$ is $M$-generic for $\overline\p_{\overline\kappa}*\dot\q_{\overline\kappa}$, where $\dot\q_{\overline\kappa}$ is the $\Add(\overline\kappa,\overline\theta)$ of $M^{\overline\p_{\overline\kappa}}$. Thus, the lifting criterion is satisfied by construction, and so $\pi$ lifts to
\begin{displaymath}
\pi:M[\overline G][\overline g]\to H_\lambda[G][g]
\end{displaymath}
in $V[G][g]$. It remains to argue that we can lift $\sigma$ to $M[\overline G][\overline g]$ so that $\pi$ and $\sigma$ continue to constitute a $(\kappa,\lambda)$-remarkable pair.

First, we lift $\sigma$ to $M[\overline G]$. As we noted earlier, $\sigma(\overline\p_{\overline\kappa})=\overline\p_{\kappa'}$ and $\overline\p_{\kappa'}$ has $\overline\p_{\overline\kappa}*\dot\q_{\overline\kappa}$ as an initial segment. Since $N$ is countable, we can choose some $N$-generic $G'$ for $\overline\p_{\kappa'}$ extending $\overline G*\overline g$ and, use the lifting criterion to lift $\sigma$ to
\begin{displaymath}
\sigma:M[\overline G]\to  N[G']
\end{displaymath}
in $V[G][g]$.

Next, we lift $\sigma$ to $M[\overline G][\overline g]$. Let $\q_{\overline\kappa}=\Add(\overline\kappa,\overline\theta)^{M[\overline G]}$, let $\q_{\kappa'}=\Add(\kappa',\sigma(\overline\theta))^{N[G']}$, and let $g'$ be any $N[G']$-generic for $\q_{\kappa'}$, which exists because $N[G']$ is countable.

Recall that conditions in a poset of the form $\Add(\kappa,\theta)$ are partial functions $p\from\kappa\times\theta\to 2$ with domain of size less than $\kappa$. Given a function $p\from\kappa\times\theta\to 2$, we shall call the set $\{\xi\mid \exists \beta\,(\beta,\xi)\in\text{dom}(p)\}$ the \emph{support} of $p$, denoted by $\supp(p)$, and given a fixed $\xi$ in the support of $p$, we shall let $p_\xi\from\kappa\to 2$ be the function defined by $p_\xi(\beta)=p(\beta,\xi)$. Let
\begin{displaymath}
P=\Union \sigma\image \overline g.
\end{displaymath}
Clearly $P\from\kappa'\times\sigma(\overline\theta)\to 2$, but there is no reason to suppose that it is an element of $\q_{\kappa'}$, and so we cannot do the standard master condition argument. Instead, we will apply Woodin's surgery technique (see \cite{cummings:handbook} or \cite{golshani:surgery} for a more thorough presentation), which replaces the generic $g'$ with a `surgically altered' version that contains $\sigma\image \overline g$. Given $p\in \q_{\kappa'}$, let $p^*$ be the result of altering $p$ to agree with $P$. More specifically, $p^*$ has the same domain as $p$ and
\[ p^*(\beta,\xi) = \begin{cases}
      P(\beta,\xi) & \textrm{ if $(\beta,\xi)\in\text{dom}(P)$} \\
      p(\beta,\xi) & \textrm{ otherwise}. \\
   \end{cases} \]
 We will argue, in a moment, that each $p^*\in \q_{\kappa'}$, and moreover $g^*=\{p^*\mid p\in g'\}$ is $N[G']$-generic for $\q_{\kappa'}$. So let's assume this and finish the lifting argument. Fix $p\in \overline g$. Since every element of $g^*$ is compatible with $\sigma(p)$, it follows that $\sigma(p)\in g^*$. Thus, $g^*$ satisfies the lifting criterion, and so we can lift $\sigma$ to
\begin{displaymath}
\sigma:M[\overline G][\overline g]\to N[G'][g^*]
\end{displaymath}
in $V[G][g]$.
The ordinal $\alpha$ remains a regular cardinal in $N[G'][g^*]$ because $\overline\p_{\kappa'}$ has no forcing in the interval $(\overline\kappa,\alpha]$. So it remains to argue that $M[\overline G][\overline g]=H_\alpha^{N[G'][g^*]}$. But this follows from the fact that $\overline G*\overline g$ is an initial segment of $G'$, and that all subsequent forcing in $\overline\p_{\kappa'}$ occurs after stage $\alpha$.

Now we explain the details of the surgery argument. Since $\overline g$ is $M[\overline G]$-generic, the support of $\overline P=\Union \overline g$ is $\overline\theta$. It follows that the support of $P$ is $\sigma\image\overline\theta$ and $P_{\sigma(\xi)}=\overline P_\xi$. Now let's fix $p\in\q_{\kappa'}$ and argue that $p^*\in \q_{\kappa'}$. Recall that $\sigma:M\to N$ is a remarkable extender embedding with
\begin{displaymath}
N=\{\sigma(f)(a)\mid a\in S^{\lt\omega},\,f\in M\},
\end{displaymath}
where $S=V_{\kappa'}^N\union\{\kappa'\}$. It follows from a standard argument about lifts of extender embeddings that therefore
\begin{displaymath}
N[G']=\{\sigma(f)(a)\mid a\in S^{\lt\omega},\,f\in M[\overline G]\}.
\end{displaymath}
Let $\overline S=V_{\overline\kappa}^M\union \{\overline\kappa\}$ and note that $\overline S$ has size $\overline\kappa$ in $M[\overline G]$. Thus, $p=\sigma(f)(a)$ for some $f\in M[\overline G]$ and $a\in S^{\lt\omega}$, and we can assume that $f:\overline S^{\lt\omega}\to \q_{\overline\kappa}$. Let's consider the intersection of the supports of $P$ and $p$. If $\sigma(\xi)$ is in the support of $p$, then by elementarity, $\xi$ must be in the support of $f(x)$ for some $x\in \overline S^{\lt\omega}$. Let
\begin{displaymath}
I=\Union_{x\in\overline S^{\lt\omega}}\text{supp}(f(x))
\end{displaymath}
be the union of the supports of all elements in the range of $f$. Because the domain of $f$ has size $\overline\kappa$ and each $f(x)$ has support less than $\overline\kappa$, it follows that $|I|=\overline\kappa$ in $M[\overline G]$. Thus, $\sigma\image I$ is an element of $N[G']$ by Proposition~\ref{prop:cpSizeImage}. Now observe that to obtain $p^*$, we just need to alter $p$ to agree with $P$ on the part of its support that is contained in $\sigma\image I$. Since both $\sigma\image I$ and $\overline P$ are elements of $N[G']$, it follows that so is $p^*$. Finally, we argue that $g^*$ is $N[G']$-generic for $\q_{\kappa'}$. It is clear that $g^*$ is a filter. Fix a maximal antichain $A$ of $\q_{\kappa'}$ in $N[G']$ and let $A=\sigma(f)(a)$ for some $f:\overline S^{\lt\omega}\to \mathcal A$ in $M[\overline G]$ and $a\in S^{\lt\omega}$, where $\mathcal A$ is the collection of all antichains of $\q_{\overline\kappa}$. Let
\begin{displaymath}
J=\Union_{p\in A}\supp(p)
\end{displaymath}
be the union of the supports of all conditions in $A$ and let
\begin{displaymath}
\overline J=\Union_{p\in f(x),\,x\in \overline S^{\lt\omega}}\supp(p)
\end{displaymath}
be the union of the supports of all conditions in all antichains in the range of $f$. Since $\q_{\overline\kappa}$ has the $\overline\kappa^+$-cc in $M[\overline G]$, it follows that $\overline J$ has size $\overline\kappa$ in $M[\overline G]$. Let's consider the intersection of $J$ and the support of $P$.  If $\sigma(\xi)$ is in $J$, then by elementarity, $\xi$ is in $\overline J$.  As before, $\sigma\image \overline J$ is an element of $N[G']$. From $\sigma\image \overline J$ and $\overline P$, $N[G']$ can construct the set
\begin{displaymath}
X=\Union_{p\in A}\text{dom}(p)\cap \text{dom}(P).
\end{displaymath}
Since $X$ has size $\overline\kappa$, $g^*\restrict X$ is in $N[G']$. Let $Y$ be the complement of $X$ in $\kappa'\times\sigma(\overline\theta)$, and note that $\q_{\kappa'}$ is naturally isomorphic to $\q_{\kappa'}^X\times \q_{\kappa'}^Y$, where $\q_{\kappa'}^X$ consists of those conditions in $\q_{\kappa'}$ whose domain is contained in $X$ and $\q_{\kappa'}^Y$ is defined similarly. Let 
\begin{displaymath}
A'=\{p\restrict Y\mid p\in A\text{ and }p \text{ is compatible with }g^*\restrict X\}.
\end{displaymath}
If $q\in \q_{\kappa'}^Y$, then $q\union g^*\restrict X$ is compatible to some $p\in A$ and clearly $p\restrict Y$ is then in $A'$. Thus, $A'$ is maximal in $\q_{\kappa'}^Y$. Since $g\restrict \q_{\kappa'}$ is $N[G']$-generic for $\q_{\kappa'}$, it follows that there is $q\in g\cap A'$. Let $q=p\restrict Y$, where $p\in A$ and $p$ is compatible with $g^*\restrict X$. Since $q$ and $p$ are compatible and $A$ is an antichain, it follows that $p\in g$. But then $p=p^*$ is in $g^*$. Thus, $A\cap g^*\neq \emptyset$. 

We can now conclude that $V[G][g]$ has a $(\kappa,\lambda)$-remarkable pair of embeddings
\begin{displaymath}
\pi:M[\overline G][\overline g]\to H_\lambda[G][g]
\end{displaymath}
and
\begin{displaymath}
\sigma:M[\overline G][\overline g]\to N[G'][g^*],
\end{displaymath}
which contradicts that $q\in G*g$ forces that no such embeddings exist.

\end{proof}
\subsection{Indestructibility by all $\ltkappa$-closed $\lesseq\kappa$-distributive forcing}
Gitik and Shelah showed that strong cardinals can be made indestructible by all weakly $\lesseq\kappa$-closed forcing with the Prikry property, a class which, in particular, includes all $\lesseq\kappa$-closed forcing. Here, we prove a result along similar lines for remarkable cardinals.
\begin{theorem}\label{th:distributiveIndestructible}
If $\kappa$ is remarkable, then there is a forcing extension in which the remarkability of $\kappa$ becomes indestructible by all $\ltkappa$-closed $\lesseq\kappa$-distributive forcing.
\end{theorem}
\begin{proof}
Let's fix a remarkable Laver function $\ell_W$. Let $\p_\kappa$ be the $\kappa$-length Easton support iteration which, at stage $\xi$, forces with $\dot\q_\xi$ whenever $\ell_W(\xi)=\la \dot\q_\xi,x\ra$ for some set $x$, such that $\dot\q_\xi$ is a $\p_\xi$-name for a $\lt\xi$-closed $\lesseq\xi$-distributive poset in $V^{\p_\xi}$. It is easy to see that $\p_\kappa\subseteq V_\kappa$ and $\p_\kappa$ preserves all inaccessible cardinals. We will argue that $\kappa$ has the desired indestructibility in any forcing extension by $\p_\kappa$. Note that it suffices to argue that $\kappa$ remains remarkable in any forcing extension by $\p_\kappa*\dot \q$, where $\dot \q$ is a $\p_\kappa$-name for a $\ltkappa$-closed $\leq\kappa$-distributive poset in $V^{\p_\kappa}$ (because $\dot \q$ can name a trivial poset).

Fix a $\p_\kappa$-name $\dot \q$ for a $\ltkappa$-closed $\lesseq\kappa$-distributive poset in $V^{\p_\kappa}$ and suppose towards a contradiction that $\kappa$ is no longer remarkable in some forcing extension by $\p_\kappa*\dot \q$. Then there is a regular $\lambda>\kappa$ and a condition $q\in \p_\kappa*\dot\q$ forcing that there is no $(\kappa,\lambda)$-remarkable pair of embeddings $\pi$ and $\sigma$. We can assume without loss of generality that $\lambda$ is much larger than the cardinality of the transitive closure of $\dot\q$.

Let $\delta=(2^\lambda)^+$. Let $Y$ be a countable elementary substructure of $H_\delta$ containing $q$, $\kappa$, $\lambda$ the well-order $W$, and $\dot \q$. Let $M'$ be the Mostowski collapse of $Y$. The inverse of the collapsing map is an embedding
\begin{displaymath}
\rho:M'\to H_\delta
\end{displaymath}
with $\rho(q')=q$, $\rho(\kappa')=\kappa$, $\rho(\lambda')=\lambda$, $\rho(\ell'_{W'})=\ell_W$, and $\rho(\dot\q')=\dot \q$. In $M'$, we can define the Easton support $\kappa'$-length iteration  $\overline \p_{\kappa'}$ which, at stage $\xi$, forces with $\dot\q_\xi$ whenever $\ell'_{W'}(\xi)=\la \dot\q_\xi,x\ra$ for some set $x$, such that $\dot\q_\xi$ is a $\overline \p_\xi$-name for a $\lt\xi$-closed $\leq\xi$-distributive poset in $(M')^{\overline \p_\xi}$. Clearly $\rho(\overline\p_{\kappa'})=\p_\kappa$. Since $M'$ is countable, we can choose some $M'$-generic $h$ for $\Coll(\omega,\lt\kappa')^{M'}$. By elementarity (and Lemma~\ref{le:LaverFunctionGap}), $M'[h]$ has a $(\overline\kappa,\alpha,\kappa',\lambda')$-remarkable embedding
\begin{displaymath}
\overline\sigma:H_\alpha^{M'}\to H_{\lambda'}^{M'}
\end{displaymath}
with $\overline\sigma(\overline q)=q'$, $\overline\sigma(\overline \ell)=\ell'_{W'}$, and $\overline\sigma(\dot\q_{\overline\kappa})=\dot\q'$ such that
\begin{enumerate}
\item $\text{dom}(\ell'_{W'})\cap (\overline\kappa,\alpha]=\emptyset$,
\item $\ell'_{W'}(\overline\kappa)=\la \dot\q_{\overline\kappa},x\ra$ for some set $x$,
\item $\overline \ell=\ell'_{W'}\restrict\overline\kappa$.
\end{enumerate}
In particular, we have $\overline\sigma(\overline\p_{\overline\kappa})=\overline\p_{\kappa'}$ and $\overline\p_{\kappa'}$ forces with $\dot \q_{\overline\kappa}$ at stage $\overline\kappa$. Let
\begin{displaymath}
\sigma:H^{M'}_\alpha\to N
\end{displaymath}
be the remarkable extender embedding obtained from $\overline\sigma$. Note that $H^{M'}_\alpha\subseteq N$ and $\overline\sigma$ and $\sigma$ agree on $V^{M'}_{\overline\kappa+1}$. Thus, for instance, it continues to be the case that $\sigma(\overline \ell)=\ell'_{W'}$ and $\sigma(\overline\p_{\overline\kappa})=\overline\p_{\kappa'}$. Let $M=H_\alpha^{M'}$ and
\begin{displaymath}
\pi:M\to H_\lambda,
\end{displaymath}
where $\pi=\rho\circ\overline\sigma$. Note that $\pi(\overline\kappa)=\kappa$ and $\pi(\overline q)=q$. Clearly $\pi$ and $\sigma$ is a $(\kappa,\lambda)$-remarkable pair of embeddings.

Since $\p_\kappa*\dot\q$ is, in particular, countably closed, we can find a $V$-generic filter $G*g\subseteq \p_\kappa*\dot\q$ that is $X$-generic for $X=\pi\image M$ with $q\in G*g$. Let $\overline G*\overline g$ be the pre-image of $X\cap G*g$ under $\pi$. Since $G*g$ is $X$-generic, it follows that $\overline G*\overline g$ is $M$-generic for $\overline\p_{\overline\kappa}*\dot\q_{\overline\kappa}$. Thus, the lifting criterion is satisfied by construction, and so $\pi$ lifts to
\begin{displaymath}
\pi:M[\overline G][\overline g]\to H_\lambda[G][g]
\end{displaymath}
in $V[G][g]$. It remains to argue that we can lift $\sigma$ to $M[\overline G][\overline g]$ so that $\pi$ and $\sigma$ continue to constitute a $(\kappa,\lambda)$-remarkable pair.

First, we lift $\sigma$ to $M[\overline G]$. As we noted earlier, $\sigma(\overline\p_{\overline\kappa})=\overline\p_{\kappa'}$ and $\overline \p_{\kappa'}$ forces with $\dot \q_{\overline\kappa}$ at stage $\overline\kappa$. It follows that $\overline\p_{\kappa'}$ has $\overline\p_{\overline\kappa}*\dot\q_{\overline\kappa}$ as an initial segment. Since $N$ is countable, we can choose some $N$-generic $G'$ for $\overline\p_{\kappa'}$ extending $\overline G*\overline g$ and, use the lifting criterion to lift $\sigma$ to
\begin{displaymath}
\sigma:M[\overline G]\to  N[G']
\end{displaymath}
in $V[G][g]$.

Next, we lift $\sigma$ to $M[\overline G][\overline g]$. Let $\dot\q_{\kappa'}=\sigma(\dot\q_{\overline\kappa})$ and $\q_{\kappa'}=(\dot\q_{\kappa'})_{G'}$. Let $g'=\la\sigma\image\overline g\ra$ be the filter generated by $\sigma\image \overline g$. The filter $g'$ is obviously $\sigma\image M[\overline G]$-generic. But we will argue that it is actually fully $N[G']$-generic. It suffices to show that every dense open subset of $\q_{\kappa'}$ has a dense subset in $\sigma\image M[\overline G]$. So fix a dense open subset $D$ of $\q_{\kappa'}$ in $N[G']$. Recall that $\sigma:M\to N$ is a remarkable extender embedding with
\begin{displaymath}
N=\{\sigma(f)(a)\mid a\in S^{\lt\omega},\,f\in M\},
\end{displaymath}
where $S=V_{\kappa'}^N\union\{\kappa'\}$, and so
\begin{displaymath}
N[G']=\{\sigma(f)(a)\mid a\in S^{\lt\omega},\,f\in M[\overline G]\}.
\end{displaymath}
Thus, $D=\sigma(f)(a)$ for some $f\in M[\overline G]$ and $a\in S^{\lt\omega}$. Let $\overline D$ be the intersection of all $\sigma(f)(b)$ for $b\in S^{\lt\omega}$ that are dense open in $\q_{\kappa'}$. Clearly $\overline D\subseteq D$, and $\overline D$ is an element of $\sigma\image M[\overline G]$ because it is definable from $\sigma(f)$. Also, $\overline D$ is dense in $\q_{\kappa'}$ because, in $N[G']$, $S$ has size $\kappa'$ and $\q_{\kappa'}$ is $\lesseq\kappa'$-distributive. Thus, $g'$ meets $\overline D$, and hence $D$. By the lifting criterion, we can now lift $\sigma$ to
\begin{displaymath}
\sigma:M[\overline G][\overline g]\to N[G'][g'].
\end{displaymath}
Thus, $V[G][g]$ has a $(\kappa,\lambda)$-remarkable pair of embeddings
\begin{displaymath}
\pi:M[\overline G][\overline g]\to H_\lambda[G][g]
\end{displaymath}
and
\begin{displaymath}
\sigma:M[\overline G][\overline g]\to N[G'][g'],
\end{displaymath}
which contradicts that $q\in G*g$ forces that no such embeddings exist.
\end{proof}
\subsection{Main Theorem}
Combining the proofs of Theorem~\ref{th:addCohenSubsetsIndestructible} and Theorem~\ref{th:distributiveIndestructible}, we easily obtain the main theorem.
\begin{theorem}\label{th:main}
If $\kappa$ is remarkable, then there is a forcing extension in which the remarkability of $\kappa$ becomes indestructible by all $\ltkappa$-closed $\lesseq\kappa$-distributive forcing and all two-step iterations $\Add(\kappa,\theta)*\dot{\mathbb R}$, where $\dot{\mathbb R}$ is forced to be $\ltkappa$-closed and $\leq\kappa$-distributive.
\end{theorem}
\begin{proof}
We give a sketch of the proof using the notation of the proofs of Theorem~\ref{th:addCohenSubsetsIndestructible} and Theorem~\ref{th:distributiveIndestructible}. Let's fix a remarkable Laver function $\ell_W$. Let $\p_\kappa$ be the $\kappa$-length Easton-support iteration, which at stage $\xi$, forces with $\dot \q_\xi$ whenever $\ell_W(\xi)=\la \dot\q_\xi,x\ra$ for some set $x$, such that $\dot\q_\xi$ is a $\p_\xi$-name for either a $\lt\xi$-closed $\lesseq\xi$-distributive forcing or for a forcing of the form $\Add(\xi,\mu)*\dot {\mathbb R}$, where $\dot{\mathbb R}$ is forced to be $\lt\xi$-closed and $\lesseq\xi$-distributive. We will argue that $\kappa$ has the desired indestructibility in any forcing extension by $\p_\kappa$. When we need to lift, we have $\overline\sigma(\ell'_{W'})(\kappa')=\la\dot\q',x\ra$ (where $\dot \q'$ is a $\overline{\p}_{\kappa'}$-name for a poset in one of our two classes), and then by elementarity, $\ell'_{W'}(\overline\kappa)=\la\dot\q_{\overline\kappa},x\ra$, where $\dot\q_{\overline\kappa}$ is in the same class as $\dot\q'$. If $\dot\q'$ is a name for a $\ltkappa'$-closed $\lesseq\kappa'$-distributive forcing, then the argument uses the proof of Theorem~\ref{th:distributiveIndestructible}. If $\dot\q_{\kappa'}$ is a name for a poset of the form $\Add(\kappa',\mu)*\dot {\mathbb R}$, then the lifting argument is a combination of the techniques in Theorem~\ref{th:addCohenSubsetsIndestructible} and Theorem~\ref{th:distributiveIndestructible}: for the remarkable extender embedding $\sigma: M[\overline G]\to N[G']$, we can first lift it using the surgery technique to $\sigma: M[\overline G][\overline g]\to N[G'][g']$ as in the proof Theorem~\ref{th:addCohenSubsetsIndestructible} and then lift it further by the  argument of the proof of Theorem~\ref{th:distributiveIndestructible}.
\end{proof}
\subsection{Indestructibility by the canonical forcing of the $\GCH$}
Recall that the canonical forcing of the $\GCH$ is the $\ORD$-length Easton-support iteration $\p$ that forces with $\Add(\xi^+,1)^{V^{\p_\xi}}$ at stages $\xi$, whenever $\xi$ is a cardinal in $V^{\p_\xi}$.
\begin{theorem}
If $\kappa$ is remarkable, then its remarkability is preserved in any forcing extension by the canonical forcing of the $\GCH$.
\end{theorem}
\begin{proof}
Observe that if $G\subseteq\p$ is $V$-generic and $\lambda$ is a cardinal in $V[G]$, then $H_\lambda^{V[G]}=H_\lambda^{V[G_\lambda]}$, where $G_\lambda$ is $V$-generic for the initial segment $\p_\lambda$ of $\p$. Thus, fixing $\lambda>\kappa$, it suffices to show that whenever $\lambda$ is a regular cardinal in some forcing extension $V[G_\lambda]$ by $\p_\lambda$, then $V[G_\lambda]$ has a $(\kappa,\lambda)$-remarkable pair of embeddings. So let's suppose towards a contradiction there is a $\lambda$ such that some condition $q\in \p_\lambda$ forces that $\lambda$ is regular and there is no $(\kappa,\lambda)$-remarkable pair of embeddings. Fix a cardinal $\theta\!>\!>\!\lambda$. Observe that $\p_\lambda$ is countably closed and it factors as $\p_\kappa*\p_\tail$, where $\p_\kappa\subseteq V_\kappa$ and $\p_\tail$ is $\leq\kappa$-closed. Following the proof of Theorem~\ref{th:distributiveIndestructible} and adopting its notation, we obtain a $(\kappa,\theta)$-remarkable pair of embeddings
\begin{displaymath}
\pi:M\to H_\theta\text{ and }\sigma:M\to N
\end{displaymath}
in $V$ such that $\pi(\overline\kappa)=\kappa$, $\pi(\overline\lambda)=\lambda$, $\pi(\overline q)=q$, and
\begin{displaymath}
N=\{\sigma(f)(a)\mid a\in S^{\lt\omega},\,f\in M\},
\end{displaymath}
where $S=V_{\sigma(\overline\kappa)}^N\union\{\sigma(\overline\kappa)\}$. By choosing a $V$-generic $G\subseteq \p_\lambda$ containing $q$ that is also $\pi\image M$-generic, we lift $\pi$ to
\begin{displaymath}
\pi:M[\overline G]\to H_\theta[G],
\end{displaymath}
where $\overline G\subseteq \overline\p_{\overline\lambda}$, the canonical forcing of the $\GCH$ up to $\overline\lambda$ from the perspective of $M$, and we lift $\sigma$ to
\begin{displaymath}
\sigma:M[\overline G_{\overline\kappa}][\overline G_\tail]\to N[G'_{\sigma(\overline\kappa)}][G'_\tail],
\end{displaymath}
where $\overline G$ factors as $\overline G_{\overline\kappa}*\overline G_\tail$ and $\overline G\subseteq G'_{\sigma(\overline\kappa)}$. Since $q$ forces that $\lambda$ remains regular, it follows that $\overline q$ forces that $\overline\lambda$ remains regular, meaning that $\overline\lambda$ is regular in $M[\overline G]$, and therefore must remain regular in $N[G']$, where $G'=G'_{\sigma(\overline\kappa)}*G'_\tail$, since the forcing beyond $\overline\lambda$ is sufficiently closed. Also for this reason, $H_{\overline\lambda}^{M[\overline G]}=H_{\overline\lambda}^{N[G']}$. Thus, the restrictions
\begin{displaymath}
\pi:H^{M[\overline G]}_{\overline\lambda}\to H_{\lambda}^{V[G]}
\end{displaymath}
and
\begin{displaymath}
\sigma:H^{M[\overline G]}_{\overline\lambda}\to H_{\sigma(\overline\lambda)}^{N[G']}
\end{displaymath}
is a $(\kappa,\lambda)$-remarkable pair of embeddings. Thus, we have produced a $(\kappa,\lambda)$-remarkable pair of embeddings in a forcing extension $V[G]$ with $q\in G$, which is the desired contradiction.
\end{proof}
\section{Applications of indestructibility}\label{sec:applications}
In this section, we give two applications of the indestructibility provided by the main theorem. We show that it is consistent to realize any possible continuum pattern above a remarkable cardinal. Using techniques developed recently in~\cite{ChengFriedmanHamkins:LargeCardinalsNeedNotBeLargeInHOD}, we show that it is consistent to have a remarkable cardinal that is not remarkable, and indeed not even weakly compact in ${\rm HOD}$.

Let's define that a (possibly partial) class function $F$ on the regular cardinals is an \emph{Easton} function if for all $\alpha<\beta$ in the domain of $F$, we have $F(\alpha)\leq F(\beta)$ and for all $\alpha$ in the domain of $F$, we have $\text{cf}(F(\alpha))>\alpha$. By the Zermelo-\Konig\ inequality, any continuum pattern on the regular cardinals is described by an Easton function.
\begin{theorem}
Suppose that $\kappa$ is remarkable and the $\GCH$ holds. If $F$ is any Easton function with
\hbox{$\text{dom}(F)\cap \kappa=\emptyset$}, then there is a class forcing extension in which $\kappa$ remains remarkable and for all $\alpha\in\text{dom}(F)$, we have $2^\alpha=F(\alpha)$.
\end{theorem}
\begin{proof}
By doing a preparatory forcing, if necessary, we can assume that the remarkability of $\kappa$ is indestructible by all $\ltkappa$-closed $\lesseq\kappa$-distributive forcing and all two-step iterations $\Add(\kappa,\theta)*\dot{\mathbb R}$, where $\dot{\mathbb R}$ is forced to be $\ltkappa$-closed and $\lesseq\kappa$-distributive (Theorem~\ref{th:main}). Let $\p$ be the Easton-support product
\begin{displaymath}
\prod_{\alpha\in\text{dom}(F)}\Add(\alpha,F(\alpha))
\end{displaymath}
and let $G\subseteq\p$ be $V$-generic. A classical argument of Easton shows that the continuum function on the regular cardinals $\alpha\geq\kappa$ agrees with $F$. So it remains to argue that $\kappa$ is remarkable in $V[G]$. Observe that every set-length initial segment
\begin{displaymath}
\p_\gamma=\prod_{\alpha\in\text{dom}(F),\alpha<\gamma}\Add(\alpha,F(\alpha))
\end{displaymath}
of $\p$ ($\gamma>\kappa$) factors as
\begin{displaymath}
\Add(\kappa,F(\kappa))\times \q,
\end{displaymath}
where $\q$ is forced to be $\ltkappa$-closed and $\lesseq\kappa$-distributive. The poset $\q$ is $\lesseq\kappa$-distributive after forcing with $\Add(\kappa,F(\kappa))$ by Easton's Lemma\footnote{Easton's Lemma states that if $\p$ has the $\kappa^+$-cc and $\q$ is ${\leq}\kappa$-closed, then $\q$ remains $\lesseq\kappa$-distributive after forcing with $\p$.}, and since $\q$ is $\ltkappa$-closed in $V$, it remains $\ltkappa$-closed after forcing with $\Add(\kappa,F(\kappa))$ because that poset does not add sequences of length $\ltkappa$. Thus, by our indestructibility assumption, $\kappa$ is remarkable in every $V[G_\gamma]$, where $G_\gamma$ is the restriction of $G$ to $\p_\gamma$, and the existence of a $(\kappa,\lambda)$-remarkable pair of embeddings cannot be destroyed by sufficiently closed forcing.
\end{proof}
Next, we sketch the argument that a remarkable cardinal need not be remarkable in ${\rm HOD}$. For details, we refer the reader to~\cite{ChengFriedmanHamkins:LargeCardinalsNeedNotBeLargeInHOD}.
\begin{theorem}
If $\kappa$ is remarkable, then there is a forcing extension in which $\kappa$ is remarkable, but not weakly compact in ${\rm HOD}$.
\end{theorem}
\begin{proof}
By doing a preparatory forcing, if necessary, we can assume that the remarkability of $\kappa$ is indestructible by all $\ltkappa$-closed $\lesseq\kappa$-distributive forcing and all two-step iterations $\Add(\kappa,\theta)*\dot{\mathbb R}$, where $\dot{\mathbb R}$ is forced to be $\ltkappa$-closed and $\lesseq\kappa$-distributive (Theorem~\ref{th:main}). Let $\q$ be the forcing to add a homogeneous $\kappa$-Souslin tree and consider the two-step iteration $\q* \dot T$, where $\dot T$ is the canonical $\q$-name for the $\kappa$-Souslin tree added by $\q$. It is a classical result that this two-step iteration is forcing equivalent to $\Add(\kappa,1)$~\cite{kunen:weakcompacts_are_not_downward_absolute}. Let $T*b\subseteq \q*\dot T$ be $V$-generic.
In $V[T]$, we force with the standard $\GCH$ coding forcing $\mathbb R$ to code $T$ into the continuum pattern above $\kappa$ and let $H\subseteq \mathbb R$ be $V[T]$-generic. Note that the $V[T]$-generic $b$ is also $V[T][H]$-generic because $T$ is still a $\kappa$-Souslin tree in $V[T][H]$ (because $\mathbb R$ is $\lesseq\kappa$-distributive) and every branch of a $\kappa$-Souslin tree is generic. Thus, $H$ and $b$ are mutually generic, giving us that $V[T][H][b]=V[T][b][H]$. Also, since $T$ has size $\kappa$, and so obviously has the $\kappa^+$-cc, $\mathbb R$ is $\lesseq\kappa$-distributive in $V[T][b]$ by Easton's Lemma. Finally observe that $\mathbb R$ remains $\ltkappa$-closed in $V[T][b]$ because the forcing $T$ is $\lt\kappa$-distributive. Thus, the combined forcing can be viewed as a two-step iteration of the form $\Add(\kappa,1)*\dot {\mathbb R}$, where $\dot{\mathbb R}$ is forced by $\Add(\kappa,1)$ to be $\ltkappa$-closed and $\lesseq\kappa$-distributive. So, by our indestructibility assumption, $\kappa$ is remarkable in $V[T][b][H]$, but we will argue that it is not weakly compact in ${\rm HOD}^{V[T][b][H]}$. Since $V[T][b][H]=V[T][H][b]$ and the forcing $T$ is weakly homogeneous in $V[T][H]$, ${\rm HOD}^{V[T][b][H]}\subseteq V[T][H]$. The tree $T$ is an element of ${\rm HOD}^{V[T][b][H]}$ because it was coded into the continuum pattern, and it is $\kappa$-Souslin there because it is $\kappa$-Souslin in $V[T][H]$. Thus, $\kappa$ is not weakly compact in ${\rm HOD}^{V[T][b][H]}$.
\end{proof}
\section{Questions}\label{sec:questions}
In this article, we showed that remarkable cardinals have Laver-like functions and adapted techniques for making strong cardinals indestructible to the context of remarkable cardinals. It is therefore reasonable to think that remarkable cardinals can be made indestructible by all weakly $\lesseq\kappa$-\emph{closed forcing with the Prikry property}, the class of forcing notions by which Gitik and Shelah showed that strong cardinals can be made indestructible.\footnote{Consider triples of the form $\la \p,\leq,\leq_*\ra$, where $\p$ is a set and $\leq_*$, $\leq$ are two partial orders on $\p$ such that $\leq_*\subseteq \leq$, meaning that $p\leq_*q\longrightarrow p\leq q$ for every $p,q\in \p$. We will force with $\la \p,\leq\ra$. Such a triple is said to have the \emph{Prikry property} if for every $p\in \p$ and every statement $\varphi$ of the forcing language, there exists $q\leq_*p$ such that $q\forces \varphi$ or $q\forces \neg\varphi$. Let $\alpha$ be a cardinal. Such a triple is said to be $\lesseq\alpha$-\emph{weakly closed} if $\leq_*$ is $\lesseq\alpha$-closed.} All $\lesseq\kappa$-closed forcing are $\lesseq\kappa$-weakly closed with the Prikry property (with $\leq_*$ defined to be same as $\leq$) and it is not difficult to see that all $\lesseq\kappa$-weakly closed forcing with the Prikry property are $\lesseq\kappa$-distributive.
\begin{question}
Can a remarkable $\kappa$ be made indestructible by all weakly $\lesseq\kappa$-closed forcing with the Prikry property?
\end{question}
It is also feasible that remarkable cardinals have much stronger indestructibility properties, akin to those of strongly unfoldable cardinals, as shown by Johnstone and Hamkins~\cite{hamkinsjohnstone:unfoldable}.
\begin{question}\label{ques:fullIndestructibility}
Can a remarkable $\kappa$ be made indestructible by all $\ltkappa$-closed $\kappa^+$-preserving forcing?
\end{question}
\noindent It should be noted that the indestructibility requested in Question~\ref{ques:fullIndestructibility} is optimal because the existence of a weakly compact cardinal $\kappa$ that remains weakly compact in some forcing extension by $\ltkappa$-closed forcing that collapses $\kappa^+$ implies the failure of $\square_\kappa$ in $V$, and therefore has the strength of at least infinitely many Woodin cardinals (see \cite{hamkinsjohnstone:unfoldable} for details).

\bibliographystyle{alpha}
\bibliography{remarkable}
\end{document}